\documentclass[reqno,tbtags,intlimits,a4paper,oneside,12pt]{amsart}
\usepackage[cp1251]{inputenc}%
\usepackage[english]{babel}
\usepackage{amssymb,upref,calc,url}
\usepackage[mathscr]{eucal}
\usepackage{graphicx}
\usepackage{amsmath,amscd,amsthm,verbatim}
\usepackage[all]{xy}
\usepackage[in]{fullpage}
\newtheorem{thm}{\hskip\parindent Theorem}[section]

\begin{document}

\title{Polynomial vector fields in $\mathbb{C}^\infty$ determining differentiation of hyperelliptic functions of~any~genus}
\author{E.\,Yu.~Bunkova}
\address{Steklov Mathematical Institute of Russian Academy of Sciences, Moscow, Russia}
\email{bunkova@mi-ras.ru}

\begin{abstract}
In this work we give direct proofs of two theorems concerning explicitly defined polynomial vector fields connected to differentiation of hyperelliptic functions of~any~genus. We prove that the operators determining the fields commute, and we show that each of them annulates polynomials defined in terms of generating functions in $\mathbb{C}^\infty$.
\end{abstract}
\maketitle

\section{Introduction} \label{S0}

In the work \cite{BB24} polynomial dynamical systems with operators $\mathcal{D}_{k}$ in the infinite-dimensional complex space $\mathbb{C}^{\infty}$ are defined. The indices $k$ of the operators $\mathcal{D}_{k}$ are odd natural numbers.

The complex space $\mathbb{C}^{\infty}$ has coordinates $b = (b_{1,2j-1}, b_{2,2j-1}, b_{3,2j-1} $ for $ j \in \mathbb{N})$ and a~direct limit topology induced from $\mathbb{C}^{\infty} = \lim_{\rightarrow} \mathbb{C}^{3g}$ for complex spaces $\mathbb{C}^{3g}$ with coordinates $b^g = (b_{1,2j-1}, b_{2,2j-1}, b_{3,2j-1}$ for $j = 1, \ldots, g)$ and $g \in \mathbb{N}$. One can find the details in \S 4 in~\cite{BB24}. 

The operators $\mathcal{D}_{k}$ define polynomial vector fields in $\mathbb{C}^{\infty}$ by the formulas (40)--(42) in~\cite{BB24}, namely for odd natural numbers $j$ and $k$:
\begin{align}
\mathcal{D}_1(b_{1, j}) &= b_{2, j}, \label{D1} &
\mathcal{D}_1(b_{2, j}) &= b_{3, j}, &
\mathcal{D}_1(b_{3, j}) &= 4 (2 b_{1, 1} b_{2, j} + b_{2, 1} b_{1, j} + b_{2, j+2}),
\end{align}
\begin{align}
\mathcal{D}_k(b_{1, j}) &=
b_{2, j+k-1} + \sum_{s=1}^{(k-1)/2} b_{2, 2s-1} b_{1, j+k-2s-1} - \sum_{s=1}^{(k-1)/2}  b_{1, 2s-1} b_{2, j+k-2s-1},
\label{D2}
\end{align}
\begin{align}
\mathcal{D}_k(b_{2, j}) &= \mathcal{D}_1 (\mathcal{D}_k(b_{1, j})), &
\mathcal{D}_k(b_{3, j}) &= \mathcal{D}_1 (\mathcal{D}_k(b_{2,j})). \label{D3}
\end{align}

The corresponding to $\mathcal{D}_{k}$ operators in $\mathbb{C}^{3g}$ give differentiations of genus $g$ hyperelliptic functions with respect to variables in the fiber of the universal bundle of Jacobians of~genus~$g$ hyperelliptic curves, see \cite{BB24} for the correspondence and~\cite{BBN24} and \cite{BB23} for the definitions and constructions.

Section \ref{S1} of this work is devoted to Theorem \ref{comthm} that states that all the operators $\mathcal{D}_{k}$ commute. The result of this theorem is directly related to Corollary 4.4 in~\cite{BB24}. Here we give a proof of Theorem \ref{comthm} that uses only the explicit formulas \eqref{D1}--\eqref{D3}.

In Theorem 6.7 and eq. (54) of the work \cite{BB22} one can find relations that give an expression of some parameters $\lambda_{2j+2}$ where $j \in \mathbb{N}$ as polynomials in the coordinates $b$. The corresponding relations in $\mathbb{C}^{3g}$ are given in Theorem 4.2 in \cite{BS} and in this case the parameters $\lambda_{2j+2}$ for $j \leqslant 2g$ are parameters of a hyperelliptic curve, see Lemma 5.3 in~\cite{BB22}. The relations are the following:

\begin{equation} \label{L1}
 4 {\bf m}(\xi) = {\bf b}_2(\xi)^2 + 2 {\bf b}_3(\xi) (1 - {\bf b}_1(\xi)) + 4 (\xi^{-1} + 2 b_{1,1}) (1 - {\bf b}_1(\xi))^2, 
\end{equation}
where 
\begin{align*}
{\bf b}_i(\xi) &= \sum_{j=1}^\infty b_{i,2j-1} \xi^j & & \text{for} &  i &= 1,2,3, & & \text{and} &
{\bf m}(\xi) &= \xi^{-1} + \sum_{j=1}^\infty \lambda_{2j+2} \xi^j,
\end{align*}
and $\xi$ is a free parameter, so the equation \eqref{L1} gives relations expressing the parameters~$\lambda_{2j+2}$ in its left hand side as polynomials of the coordinates $b$ in its right hand side, see \cite{BB22} for more details.

Section \ref{S2} of this work is devoted to Theorem \ref{invthm} that states that we have $\mathcal{D}_{k}(\lambda_{2j+2}) = 0$ where the $\lambda_{2j+2}$ are polynomials in the coordinates $b$ defined by the relations \eqref{L1}. Therefore these polynomials give integrals of the system \eqref{D1}--\eqref{D3}.  The result of this theorem is directly related to Corollary 4.5 in~\cite{BB24}. Here we give a proof of Theorem \ref{invthm} that uses only the explicit formulas \eqref{D1}--\eqref{D3} and the relations \eqref{L1}.

\vfill

\eject

\section{The operators $\mathcal{D}_{k}$ commute} \label{S1}

From the relations \eqref{D3} using the relations \eqref{D1}--\eqref{D2} we get the explicit formulas
\begin{align}
\mathcal{D}_k(b_{2,j}) &=
b_{3,j+k-1} + \sum_{s=1}^{(k-1)/2} b_{3,2s-1} b_{1, j+k-2s-1} - \sum_{s=1}^{(k-1)/2} b_{1,2s-1} b_{3,j+k-2s-1},
\label{D4}
\\
\mathcal{D}_k(b_{3,j}) &= 
4 (2 b_{1,1} b_{2,j+k-1} + b_{2,1} b_{1,j+k-1} + b_{2,j+k+1}) + \nonumber \\
& \qquad + 4 \sum_{s=1}^{(k-1)/2} (2 b_{1, 1} b_{2, 2s-1} + b_{2, 1} b_{1, 2s-1} + b_{2, 2s+1}) b_{1, j+k-2s-1} + \nonumber \\ & \qquad  + \sum_{s=1}^{(k-1)/2} b_{3,2s-1} b_{2,j+k-2s-1} - \sum_{s=1}^{(k-1)/2} b_{2,2s-1} b_{3,j+k-2s-1} - \nonumber \\
& \qquad - 4 \sum_{s=1}^{(k-1)/2} b_{1,2s-1} (2 b_{1, 1} b_{2, j+k-2s-1} + b_{2, 1} b_{1, j+k-2s-1} + b_{2, j+k-2s+1}).
\label{D5}
\end{align}

\begin{thm} \label{comthm}
For odd natural numbers $k,l$ we have
\begin{align*}
[\mathcal{D}_k,\mathcal{D}_l] &= 0,
\end{align*}
where the brackets denote the commutator of operators.
\end{thm}

\begin{proof}

First let us prove that
\begin{equation} \label{D6}
[\mathcal{D}_1,\mathcal{D}_k] = 0
\end{equation}
for every odd natural number $k$.
By formulas \eqref{D1} and \eqref{D3} we have
\begin{align*}
\mathcal{D}_k(\mathcal{D}_1(b_{1,j})) &= \mathcal{D}_1(\mathcal{D}_k(b_{1,j})), & \mathcal{D}_k(\mathcal{D}_1(b_{2,j})) &= \mathcal{D}_1(\mathcal{D}_k(b_{2,j})).
\end{align*}
Note that for $j = 1$ from \eqref{D2} we have $\mathcal{D}_k(b_{1, 1}) =
b_{2, k}$ because
\[
 \sum_{s=1}^{(k-1)/2} b_{2, 2s-1} b_{1, k-2s} = \sum_{s=1}^{(k-1)/2}  b_{1, 2s-1} b_{2, k-2s}.
\]
From \eqref{D4} we have $\mathcal{D}_k(b_{2, 1}) =
b_{3, k}$ because
\[
 \sum_{s=1}^{(k-1)/2} b_{3,2s-1} b_{1, k-2s} = \sum_{s=1}^{(k-1)/2} b_{1,2s-1} b_{3,k-2s}.
\]
Therefore from the expressions \eqref{D1}, \eqref{D2} and \eqref{D4} we get
\begin{multline}
{1 \over 4} \mathcal{D}_k(\mathcal{D}_1(b_{3,j})) = \mathcal{D}_k(2 b_{1,1} b_{2,j} + b_{2,1} b_{1,j} + b_{2,j+2}) = \\
= 2 b_{2,k} b_{2,j} + 2 b_{1,1} \mathcal{D}_k(b_{2,j}) +  b_{3,k} b_{1,j} + b_{2,1} \mathcal{D}_k(b_{1,j}) +  \mathcal{D}_k(b_{2,j+2}) = \\
= 2 b_{2,k} b_{2,j} + 2 b_{1,1} b_{3,j+k-1} + 2 b_{1,1} \sum_{s=1}^{(k-1)/2} b_{3,2s-1} b_{1, j+k-2s-1} - 2 b_{1,1} \sum_{s=1}^{(k-1)/2} b_{1,2s-1} b_{3,j+k-2s-1} + \\
+  b_{3,k} b_{1,j} + b_{2,1} b_{2,j+k-1} + b_{2,1} \sum_{s=1}^{(k-1)/2} b_{2,2s-1} b_{1, j+k-2s-1} - b_{2,1} \sum_{s=1}^{(k-1)/2} b_{1,2s-1} b_{2,j+k-2s-1} + \\ + b_{3,j+k+1} + \sum_{s=1}^{(k-1)/2} b_{3,2s-1} b_{1, j+k-2s+1} - \sum_{s=1}^{(k-1)/2} b_{1,2s-1} b_{3,j+k-2s+1}. \label{D7}
\end{multline}

From equations \eqref{D1} and \eqref{D5} we get
\begin{multline}
{1 \over 4} \mathcal{D}_1(\mathcal{D}_k(b_{3,j})) =
 2 b_{1,1} b_{3,j+k-1} + 3 b_{2,1} b_{2,j+k-1} + b_{3,1} b_{1,j+k-1} + b_{3,j+k+1} + \\ + \sum_{s=1}^{(k-1)/2} 2 b_{2,2s+1} b_{2,j+k-2s-1} - \sum_{s=1}^{(k-1)/2} 2 b_{2,2s-1} b_{2,j+k-2s+1} + \\
 + 2 b_{1,1} \sum_{s=1}^{(k-1)/2} \left( b_{3,2s-1} b_{1, j+k-2s-1} -  b_{1,2s-1} b_{3,j+k-2s-1} \right) + \\
 + b_{2,1} \sum_{s=1}^{(k-1)/2} \left( b_{2,2s-1} b_{1, j+k-2s-1} - b_{1,2s-1} b_{2,j+k-2s-1} \right) + \\ + \sum_{s=1}^{(k-1)/2} b_{3,2s+1} b_{1, j+k-2s-1} - \sum_{s=1}^{(k-1)/2} b_{1,2s-1} b_{3,j+k-2s+1}. \label{D8}
\end{multline}
Eq. \eqref{D8} coincides with eq. \eqref{D7} because
\begin{multline*}
3 b_{2,1} b_{2,j+k-1} + \sum_{s=1}^{(k-1)/2} 2 b_{2,2s+1} b_{2,j+k-2s-1} - \sum_{s=1}^{(k-1)/2} 2 b_{2,2s-1} b_{2,j+k-2s+1} = \\ = 2 b_{2,k} b_{2,j} + b_{2,1} b_{2, j+k-1}
\end{multline*}
and
\begin{multline*}
b_{3,1} b_{1,j+k-1} + \sum_{s=1}^{(k-1)/2} b_{3,2s+1} b_{1, j+k-2s-1} = \sum_{s=1}^{(k-1)/2} b_{3,2s-1} b_{1, j+k-2s+1} + b_{3,k} b_{1,j}.
\end{multline*}
Therefore equation \eqref{D6} is proved
for every odd natural number $k$.

Now let us use equation \eqref{D6} to prove the theorem. It is enough to prove the relation
\begin{equation} \label{D9}
\mathcal{D}_k(\mathcal{D}_l(b_{1,j})) = \mathcal{D}_l(\mathcal{D}_k(b_{1,j})).
\end{equation}

By \eqref{D2} we have
\begin{multline}
\mathcal{D}_l(\mathcal{D}_k(b_{1,j})) = \mathcal{D}_l(b_{2,j+k-1}) + \\ + \sum_{s=1}^{(k-1)/2} \mathcal{D}_l(b_{2,2s-1}) b_{1, j+k-2s-1} + \sum_{s=1}^{(k-1)/2} b_{2,2s-1} \mathcal{D}_l(b_{1, j+k-2s-1}) - \\ - \sum_{s=1}^{(k-1)/2} \mathcal{D}_l(b_{1,2s-1}) b_{2,j+k-2s-1} - \sum_{s=1}^{(k-1)/2} b_{1,2s-1} \mathcal{D}_l(b_{2,j+k-2s-1}).
\label{D10}
\end{multline}

By \eqref{D2} and \eqref{D4}, the expression \eqref{D10} is equal to the following expression
\begin{multline*}
b_{3,j+k+l-2} + \sum_{s=1}^{(l-1)/2} \left(b_{3,2s-1} b_{1, j+k+l-2s-2} - b_{1,2s-1} b_{3,j+k+l-2s-2}\right) + \\
+ \sum_{s=1}^{(k-1)/2}  b_{1, j+k-2s-1} \sum_{n=1}^{(l-1)/2} \left(b_{3,2n-1} b_{1, 2s+l-2n-2} - b_{1,2n-1} b_{3,2s+l-2n-2}\right)
+ 
\\
+ \sum_{s=1}^{(k-1)/2}  b_{1, j+k-2s-1} b_{3,2s+l-2} + \sum_{s=1}^{(k-1)/2} b_{2,2s-1} b_{2,j+k+l-2s-2} +\\ + \sum_{s=1}^{(k-1)/2} b_{2,2s-1} \sum_{n=1}^{(l-1)/2} \left(b_{2,2n-1} b_{1, j+k+l-2s-2n-2} - b_{1,2n-1} b_{2,j+k+l-2s-2n-2}\right) -
\\ 
- \sum_{s=1}^{(k-1)/2}  b_{2,j+k-2s-1} \sum_{n=1}^{(l-1)/2} \left(b_{2,2n-1} b_{1, 2s+l-2n-2} - b_{1,2n-1} b_{2,2s+l-2n-2}\right) - \\ - \sum_{s=1}^{(k-1)/2}  b_{2,j+k-2s-1} b_{2,2s+l-2}
- \sum_{s=1}^{(k-1)/2} b_{1,2s-1}  b_{3,j+k+l-2s-2} - \\ - \sum_{s=1}^{(k-1)/2} b_{1,2s-1} \sum_{n=1}^{(l-1)/2} \left(b_{3,2n-1} b_{1, j+k+l-2s-2n-2} - b_{1,2n-1} b_{3,j+k+l-2s-2n-2}\right). 
\end{multline*}
This relation is symmetric with respest to $k$ and $l$, because the following relations hold
\begin{multline*}
\sum_{s=1}^{(l-1)/2} b_{3,2s-1} b_{1, j+k+l-2s-2} 
+ \sum_{s=1}^{(k-1)/2}  b_{1, j+k-2s-1} b_{3,2s+l-2} = \sum_{s=1}^{(k+l-2)/2} b_{3,2s-1} b_{1, j+k+l-2s-2},
\end{multline*}
\begin{multline*}
\sum_{s=1}^{(k-1)/2} b_{2,2s-1} b_{2,j+k+l-2s-2} - \sum_{s=1}^{(k-1)/2}  b_{2,j+k-2s-1} b_{2,2s+l-2}  = \\ = \sum_{s=1}^{(l-1)/2} b_{2,2s-1} b_{2,j+k+l-2s-2} - \sum_{s=1}^{(l-1)/2}  b_{2,j+l-2s-1} b_{2,2s+k-2},
\end{multline*}
and the following expressions are symmetric with respect to $k$ and $l$:
\begin{multline}
\sum_{s=1}^{(k-1)/2}  b_{1, j+k-2s-1} \sum_{n=1}^{(l-1)/2} b_{3,2n-1} b_{1, 2s+l-2n-2} - \sum_{s=1}^{(k-1)/2}  b_{1, j+k-2s-1} \sum_{n=1}^{(l-1)/2} b_{1,2n-1} b_{3,2s+l-2n-2} - \\ - \sum_{s=1}^{(k-1)/2} b_{1,2s-1} \sum_{n=1}^{(l-1)/2} b_{3,2n-1} b_{1, j+k+l-2s-2n-2},
 \label{D11}
\end{multline}
\begin{multline}
\sum_{s=1}^{(k-1)/2}  b_{2,j+k-2s-1} \sum_{n=1}^{(l-1)/2} b_{1,2n-1} b_{2,2s+l-2n-2} - \sum_{s=1}^{(k-1)/2}  b_{2,j+k-2s-1} \sum_{n=1}^{(l-1)/2} b_{2,2n-1} b_{1, 2s+l-2n-2} -
\\  - \sum_{s=1}^{(k-1)/2} b_{2,2s-1} \sum_{n=1}^{(l-1)/2}  b_{1,2n-1} b_{2,j+k+l-2s-2n-2}. \label{D12}
\end{multline}
We see that the expressions \eqref{D11} and \eqref{D12} are related by the transform $b_{1,j} \mapsto b_{2,j}$, $b_{3,j} \mapsto b_{1,j}$, so it is sufficient to prove that one of these expressions is symmetric with respect to $k$ and $l$ to obtain the symmetry for the other one.
Let us show that \eqref{D11} is symmetric with respect to $k$ and $l$. The expression \eqref{D11} equals the following expression:
\begin{multline*}
\sum_{n=1}^{(l-1)/2} b_{3,2n-1} \sum_{s=1}^{(k-1)/2}  b_{1, j+k-2s-1} b_{1, 2s+l-2n-2} - \sum_{n=1}^{(l-1)/2} b_{3,2n-1} \sum_{s=1}^{(k-1)/2} b_{1,2s-1} b_{1, j+k+l-2s-2n-2} - \\ - \sum_{n=1}^{(l+k-4)/2} b_{3,2n-1} \sum_{s=1}^{(k-1)/2}  b_{1, j+k-2s-1} \sum_{p=1}^{(l-1)/2} b_{1,2p-1} \delta(2p-1, 2s+l-2n-2),
\end{multline*}
where $\delta(\cdot,\cdot)$ is the Kronecker delta function. Therefore the coefficient at $b_{3,2n-1}$ in \eqref{D11} for $1 \leqslant 2n-1 \leqslant min(k,l)-2$ is 
\begin{multline*}
\sum_{s=n+1}^{(k-1)/2}  b_{1, j+k-2s-1} b_{1, 2s+l-2n-2} -  \sum_{s=1}^{(k-1)/2} b_{1,2s-1} b_{1, j+k+l-2s-2n-2} = 
\\
= \sum_{s=(l+1)/2}^{(k+l)/2-n-1} b_{1, 2s-1} b_{1, j+k+l-2s-2n-2} -  \sum_{s=1}^{(k-1)/2} b_{1,2s-1} b_{1, j+k+l-2s-2n-2}.
\end{multline*}

The coefficient at $b_{3,2n-1}$ in \eqref{D11} for $l < k$ and $l \leqslant 2n-1 \leqslant k - 2$ is 
\[
-  \sum_{s=(2n-l+3)/2}^{n}  b_{1, j+k-2s-1} b_{1,2s+l-2n-2} = - \sum_{s=1}^{(l-1)/2} b_{1,2s-1} b_{1, j+k+l-2s-2n-2},
\]
while the coefficient at $b_{3,2n-1}$ in \eqref{D11} for $k < l$ and $k \leqslant 2n-1 \leqslant l - 2$ is 
\[
 -  \sum_{s=1}^{(k-1)/2} b_{1,2s-1} b_{1, j+k+l-2s-2n-2}.
\]

The coefficient at $b_{3,2n-1}$ in \eqref{D11} for $max(k,l) \leqslant 2n-1 \leqslant l + k - 5$ is 
\[
- \sum_{s=(2 n - l + 3)/2}^{(k-1)/2}  b_{1, j+k-2s-1} b_{1,2s+l-2n-2} = - \sum_{s=1}^{(k+l)/2 - n - 1}  b_{1, j+2s-2} b_{1,k+l-2s-2n-1}.
\]

These expressions are symmetric with respect to $k$ and $l$, therefore we have proved the relation \eqref{D9} and the Theorem.

\end{proof}

\vfill
\eject

\section{The operators $\mathcal{D}_{k}$ annulate  the polynomials $\lambda_{2j+2}$} \label{S2}

\begin{thm} \label{invthm}
For odd natural numbers $k$ and $j \in \mathbb{N}$ we have
\begin{align*}
\mathcal{D}_k(\lambda_{2j+2}) = 0,
\end{align*}
where $\lambda_{2j+2}$ is the polynomial in the coordinates $b$ defined by the relation \eqref{L1}.
\end{thm}

\begin{proof}
Recall that
\begin{align*}
{\bf b}_i(\xi) &= \sum_{j=1}^\infty b_{i,2j-1} \xi^j & & \text{for} &  i &= 1,2,3, & & \text{and} &
{\bf m}(\xi) &= \xi^{-1} + \sum_{j=1}^\infty \lambda_{2j+2} \xi^j,
\end{align*}
therefore from \eqref{D1} we have
\begin{align}
\mathcal{D}_1({\bf b}_1(\xi)) &= {\bf b}_2(\xi), \label{U1} \\
\mathcal{D}_1({\bf b}_2(\xi)) &= {\bf b}_3(\xi), \label{U2} \\
\mathcal{D}_1({\bf b}_3(\xi)) &= 4 (2 b_{1, 1} {\bf b}_2(\xi) + b_{2, 1} {\bf b}_1(\xi) + \xi^{-1} {\bf b}_2(\xi) - b_{2,1}). \label{U3}
\end{align}
Thus from \eqref{L1}, \eqref{U1} and \eqref{U2} we have
\begin{multline*}
2 \mathcal{D}_1({\bf m}(\xi)) = {\bf b}_2(\xi) \mathcal{D}_1({\bf b}_2(\xi)) + \mathcal{D}_1({\bf b}_3(\xi)) (1 - {\bf b}_1(\xi)) - {\bf b}_3(\xi) \mathcal{D}_1({\bf b}_1(\xi))) + \\ + 4 b_{2,1} (1 - {\bf b}_1(\xi))^2 - 4 (\xi^{-1} + 2 b_{1,1}) (1 - {\bf b}_1(\xi)) \mathcal{D}_1({\bf b}_1(\xi)) = \\
= (1 - {\bf b}_1(\xi)) \left( \mathcal{D}_1({\bf b}_3(\xi)) + 4 b_{2,1} (1 - {\bf b}_1(\xi)) - 4 (\xi^{-1} + 2 b_{1,1})  {\bf b}_2(\xi) \right),
\end{multline*}
so using the relation \eqref{U3} we get $\mathcal{D}_1({\bf m}(\xi)) = 0$, therefore $\mathcal{D}_1(\lambda_{2j+2}) = 0$ for $j \in \mathbb{N}$.

From \eqref{D2} for $j \in \mathbb{N}$, $k \in \mathbb{N}$ we obtain 
\begin{align}
\mathcal{D}_{2k-1}(b_{1, 2j-1}) &=
b_{2, 2j+2k-3} + \sum_{s=1}^{k-1} b_{2, 2s-1} b_{1, 2j+2k-2s-3} - \sum_{s=1}^{k-1}  b_{1, 2s-1} b_{2, 2j+2k-2s-3}.
\end{align}
We will use the relations
\begin{align*}
{\bf b}_i(\xi) &= \sum_{p=1}^{k-1} b_{i,2p-1} \xi^p + \sum_{j=1}^\infty b_{i,2j+2k-3} \xi^{j+k-1} & & \text{for} &  i &= 1,2,3,
\end{align*}
therefore
\begin{align*}
\sum_{j=1}^\infty b_{i,2j+2k-3} \xi^{j} &= {\bf b}_i(\xi) \xi^{-k+1} - \sum_{p=1}^{k-1} b_{i,2p-1} \xi^{p-k+1}  & & \text{for} &  i &= 1,2,3.
\end{align*}
Note that
\[
\sum_{s=1}^{k-1} b_{2, 2s-1} \sum_{p=1}^{k-s-1} b_{1,2p-1} \xi^{p-k+s+1} =
\sum_{s=1}^{k-1}  b_{1, 2s-1} \sum_{p=1}^{k-s-1} b_{2,2p-1} \xi^{p-k+s+1}.
\]
Thus 
\begin{multline*}
\mathcal{D}_{2k-1}({\bf b}_1(\xi)) = \mathcal{D}_{2k-1}\left(
\sum_{j=1}^\infty b_{1,2j-1} \xi^j\right) = \\ =
\sum_{j=1}^\infty b_{2, 2j+2k-3} \xi^j + \sum_{s=1}^{k-1} b_{2, 2s-1} \left(\sum_{j=1}^\infty b_{1, 2j+2k-2s-3} \xi^j\right) -  \sum_{s=1}^{k-1}  b_{1, 2s-1} \left(\sum_{j=1}^\infty b_{2, 2j+2k-2s-3} \xi^j\right) = \\ =
{\bf b}_1(\xi) \sum_{s=1}^{k-1} b_{2, 2s-1} \xi^{-k+s+1} + {\bf b}_2(\xi) \left(\xi^{-k+1} -  \sum_{s=1}^{k-1}  b_{1, 2s-1} \xi^{-k+s+1}\right)
- \sum_{p=1}^{k-1} b_{2,2p-1} \xi^{p-k+1}. 
\end{multline*}
 
Using the relations \eqref{D1}, \eqref{D3}, \eqref{U1} and \eqref{U2} we get
\begin{multline*}
\mathcal{D}_{2k-1}({\bf b}_2(\xi)) =
 {\bf b}_1(\xi) \sum_{s=1}^{k-1} b_{3, 2s-1} \xi^{-k+s+1} + {\bf b}_3(\xi) \left(\xi^{-k+1} -  \sum_{s=1}^{k-1}  b_{1, 2s-1} \xi^{-k+s+1} \right)
- \sum_{p=1}^{k-1} b_{3,2p-1} \xi^{p-k+1}
. 
\end{multline*}

Using the relations \eqref{D1}, \eqref{D3}, \eqref{U1} and \eqref{U3} we get
\begin{multline*}
\mathcal{D}_{2k-1}({\bf b}_3(\xi)) =
4 \left({\bf b}_2(\xi) \left(2 b_{1, 1} + \xi^{-1}\right) - b_{2,1}\right) \left(\xi^{-k+1} - \sum_{s=1}^{k-1}  b_{1, 2s-1} \xi^{-k+s+1} \right) + \\
+ 4 {\bf b}_1(\xi) \left(b_{2, 1} \xi^{-k+1} 
+ \sum_{s=1}^{k-1} (2 b_{1, 1} b_{2, 2s-1} + b_{2, 2s+1}) \xi^{-k+s+1} \right) + 
\\ + {\bf b}_2(\xi) \sum_{s=1}^{k-1} b_{3, 2s-1} \xi^{-k+s+1} - {\bf b}_3(\xi) \sum_{s=1}^{k-1}  b_{2, 2s-1} \xi^{-k+s+1} - \\
- 4 \sum_{p=1}^{k-1} (2 b_{1, 1} b_{2, 2p-1} + b_{2, 1} b_{1, 2p-1} + b_{2, 2p+1}) \xi^{p-k+1}.
\end{multline*}

Now we use these relations as well as $\mathcal{D}_{2k-1}(b_{1,1}) = b_{2,2k-1}$ to calculate the action of~$\mathcal{D}_{2k-1}$ on the equation \eqref{L1}. We have

\begin{multline*}
2 \mathcal{D}_{2k-1}({\bf m}(\xi)) = {\bf b}_2(\xi) \mathcal{D}_{2k-1}({\bf b}_2(\xi)) + \mathcal{D}_{2k-1}({\bf b}_3(\xi)) (1 - {\bf b}_1(\xi)) - {\bf b}_3(\xi) \mathcal{D}_{2k-1}({\bf b}_1(\xi)) + \\ + 4 b_{2,2k-1} (1 - {\bf b}_1(\xi))^2 - 4 (\xi^{-1} + 2 b_{1,1}) (1 - {\bf b}_1(\xi)) \mathcal{D}_{2k-1}({\bf b}_1(\xi)). 
\end{multline*}

Note that 
\begin{multline*}
{\bf b}_2(\xi) \mathcal{D}_{2k-1}({\bf b}_2(\xi)) - {\bf b}_3(\xi) \mathcal{D}_{2k-1}({\bf b}_1(\xi)) = \\ = (1 - {\bf b}_1(\xi)) \left({\bf b}_3(\xi)
 \sum_{s=1}^{k-1} b_{2, 2s-1} \xi^{-k+s+1} - {\bf b}_2(\xi) \sum_{s=1}^{k-1} b_{3, 2s-1} \xi^{-k+s+1} \right) .
\end{multline*}

Therefore
\begin{multline*}
{2 \mathcal{D}_{2k-1}({\bf m}(\xi)) \over 1 - {\bf b}_1(\xi)} = {\bf b}_3(\xi)
 \sum_{s=1}^{k-1} b_{2, 2s-1} \xi^{-k+s+1} - {\bf b}_2(\xi) \sum_{s=1}^{k-1} b_{3, 2s-1} \xi^{-k+s+1} + \mathcal{D}_{2k-1}({\bf b}_3(\xi)) + \\ + 4 b_{2,2k-1} (1 - {\bf b}_1(\xi)) - 4 (\xi^{-1} + 2 b_{1,1}) \mathcal{D}_{2k-1}({\bf b}_1(\xi)).
\end{multline*}
 Substituting the expressions for $\mathcal{D}_{2k-1}({\bf b}_1(\xi))$ and $\mathcal{D}_{2k-1}({\bf b}_3(\xi))$ in this equation, we obtain that the right hand side is equal to zero. Thus
we get $\mathcal{D}_{2k-1}({\bf m}(\xi)) = 0$, therefore $\mathcal{D}_{2k-1}(\lambda_{2j+2}) = 0$ for $j \in \mathbb{N}$, which concludes the proof of the Theorem.

\end{proof}

\vfill
\eject

\end{document}